\documentclass[letterpaper, 10 pt, conference]{ieeeconf}  

\IEEEoverridecommandlockouts                              

\usepackage{cite}
\usepackage{amsmath,amssymb,amsfonts}
\usepackage{algorithmic}
\usepackage{graphicx}
\usepackage{textcomp}
\usepackage{xcolor}
\usepackage{url}
\usepackage{array,tabularx}
\usepackage{multicol, nonfloat}  
\usepackage{epsfig} 
\usepackage{amssymb}  
\newtheorem{theorem}{Theorem}[section]

\newtheorem{lemma}[theorem]{Lemma}
\newtheorem{remark}{Remark}
\newtheorem{definition}{Definition}[section]
\newtheorem{proposition}[theorem]{Proposition}
\usepackage{blindtext}
\usepackage[utf8]{inputenc}	
\usepackage[english]{babel}
\usepackage{graphicx}
\usepackage{amsmath}
\usepackage{amsfonts}
\usepackage{epstopdf}
\usepackage{cite}
\usepackage{microtype}

\newtheorem{assumption}{Assumption}

\newcommand{\norm}[1]{\| #1 \|}

\usepackage{pgf}

\makeatletter
\newdimen\@widthOfTo%
\newdimen\@widthOfLand%
\newdimen\@widthOfImplies%
\pgfmathsetmacro{\@scaleFactorImplies}{\@widthOfTo/\@widthOfImplies}%
\pgfmathsetmacro{\@scaleFactorTo}{\@widthOfLand/\@widthOfTo}%
\newcommand*{\ScaledImplies}{\mathrel{\raisebox{0.3ex}{\scalebox{\@scaleFactorImplies}{\ensuremath{\Longrightarrow}}}}}%
\newcommand*{\ScaledTo}{\mathbin{\raisebox{0.3ex}{\scalebox{\@scaleFactorTo}{\ensuremath{\to}}}}}%
\makeatother

\title{\LARGE \bf
On the Geometry and Linear Convergence of Primal-Dual Dynamics
}

\author{P. Bansode$^{1}$, V. Chinde$^{2}$, S. R. Wagh$^{3}$, R. Pasumarthy$^{4}$, and N. M. Singh$^{3}$
\thanks{*This work was not supported by any organization}
\thanks{$^{1}$P. Bansode is with Customized Energy Solutions, Pune, India.
        {\tt\small pbansode@ces-ltd.com}}%
\thanks{$^{2}$V. Chinde is with National Renewable Energy Laboratory, Colorado, USA.
}
\thanks{$^{3}$S. R. Wagh and N. M. Singh are with Electrical Engineering Department, Veermata Jijabai Technological Institute, Mumbai, India.
}
\thanks{$^{4}$R. Pasumarthy is with Electrical Engineering Department, Indian Institute of Technology Madras, India
}}

\begin{document}

\maketitle
\thispagestyle{empty}
\pagestyle{empty}

\begin{abstract}
The paper proposes a variational-inequality based primal-dual dynamic that has a globally exponentially stable saddle-point solution when applied to solve linear inequality constrained optimization problems.
A Riemannian geometric framework is proposed wherein we begin by framing the proposed dynamics in a fiber-bundle setting endowed with a Riemannian metric that captures the geometry of the gradient (of the Lagrangian function). A strongly monotone gradient vector field is obtained by  using the natural gradient adaptation on the Riemannian manifold. The Lyapunov stability analysis proves that this adaption leads to a globally exponentially stable saddle-point solution. Further, with numeric simulations we show that the scaling a key parameter in the Riemannian metric results in an accelerated convergence to the saddle-point solution.
\end{abstract}

\section{INTRODUCTION}
\par Saddle-node dynamics has remained a subject of substantial research for many years \cite{arrow1958studies,brezzi1974existence,benzi2005numerical,nedic2009subgradient}. It iteratively seeks a solution to saddle-point problems that arise in a number of disciplines including equilibrium theory, game theory and optimization. However its application to constrained optimization problems has gained wide interest over the recent years, especially in the areas of the power networks \cite{zhao2014design,mallada2017optimal,yi2015distributed, nguyen2018contraction} and wireless networks  \cite{feijer2010stability,chen2012convergence,ferragut2014network}), and building automation systems \cite{kosaraju2018stability}, etc. Within this domain, it is popularly regarded as primal-dual dynamics. 

Over the last decade, various formulations of the primal-dual dynamics have been explored in connection with the constrained optimization problems. These algorithms were studied with primary focus on convergence of iterates to the saddle-point solution and its stability.
These algorithms either use a framework of hybrid dynamical systems \cite{feijer2010stability,cherukuri2016asymptotic} or an augmented Lagrangian technique that involves projections in the Lagrangian function \cite{dhingra2018proximal, qu2019exponential, ding2019global}. The hybrid dynamical systems approach involves switching in the dual dynamics to handle constraint violations. 
But the inherent discontinuities present in the dual dynamics make it difficult to prove exponential stability of the saddle-point solution. Thus so far, only globally asymptotic stability of the saddle-point solution has been proven.
The augmented Lagrangian approach is based on augmentation of the penalty terms in the Lagrangian function \cite{bertsekas2014constrained}. The penalty term is designed in such a way that the resulting dual dynamics is discontinuity-free, meaning that the rate of change of the dual variable does not involve switching terms. The basis of these penalty functions is either a projection or a proximal operator. So far, the latter approach has been successful in proving the global exponential stability of the algorithm when applied to solve linearly constrained optimization problems. In \cite{tang2019semi}, it is shown that the algorithm attains a semi-globally exponential convergence for a more general convex inequality constrained optimization problem.  

The algorithm that we propose in this paper uses a variational inequality based projected dynamical systems framework and produces a globally exponentially stable saddle-point solution for a linear inequality constrained optimization problem. The projected dynamical systems have been widely used to solve variational inequalities \cite{friesz1994day,nagurney2012projected,rockafellar2009variational}. We model the saddle-point problem as a variational inequality and then use the projected dynamical system that singularly handles the constraints without involving discontinuities in the dual dynamics (the detailed approach of the proposed dynamics is documented in our online report \cite{bansode2019exponential}, which is omitted from this paper due to avoid repetition). 

In contrast to the existing research on primal-dual dynamics, our algorithm does not depend on the framework of hybrid dynamical systems or the augmented Lagrangian techniques. We consider differential equations for solution trajectories of variational inequality proposed in \cite[Section 5.7.1]{friesz2010dynamic} as a basis for our algorithm, and equate the time derivative of primal-dual variables to the terms for which the corresponding fixed point problem of the variational inequality fails to be satisfied (proved in our online report \cite[Section 2.2]{bansode2019exponential}). Instead of tuning free-parameters in the augmented Lagrangian function we start by studying the geometry of gradient vector field defining the primal-dual dynamics. This enables us to carry out the required alterations to the dynamics in the Euclidean space so as to ensure linear convergence rates under suitable assumptions. We first analyze the vector field of the proposed algorithm in the Euclidean space and prove that the Euclidean geometry is not suitable for achieving the same. Then we lay down a procedure to construct a suitable differential geometry for the cause. The geometrization of the problem is based in a fiber bundle with a semi-Riemannian metric imposed by the structure of the primal-dual dynamics. A suitable alteration of the Geometry is achieved by a scaling of the connection leading to a natural gradient in the Riemannian space. As the natural gradient behaves like the Euclidean gradient in the sense that it achieves the steepest descent direction for the gradient descent algorithm \cite{amari1998natural,amari1998natural1}, the obtained natural gradient dynamics is substituted for the original Primal Dual dynamics in the Euclidean space to ensure a faster convergence rate (in our case, the contraction region is obtained in the Euclidean domain because we are using the natural gradient as opposed to the contraction region being defined by a Riemann metric). Adapting this into the proposed algorithm results in a strongly monotone gradient vector field with steepest descent and ascent directions along the primal and dual variables, respectively, then we construct a Lyapunov function which shows that for a strongly monotone gradient the proposed algorithm has an exponentially stable saddle-point solution.
Further with the help of numeric simulations, it is shown that by appropriate scaling of a key parameter in the natural gradient which in turn leads to the scaling of the connection as alluded to above, the convergence rate of the proposed algorithm can be accelerated.
\subsection*{Notations}\label{prel}
\par The set $\mathbb{R}$ (respectively $\mathbb{R}_{\geq 0}$ or $\mathbb{R}_{> 0}$) is the set of real (respectively non-negative or positive) numbers. If $f: \mathbb{R}^n \rightarrow \mathbb{R}$ is continuously differentiable in $x \in \mathbb{R}^n$, then $\nabla_x f:\mathbb{R}^n \rightarrow \mathbb{R}^n$ is the gradient of $f$ with respect to $x$. $\norm{.}$ denotes the Euclidean norm. 

\section{Problem Formulation} \label{proof} Consider the following constrained optimization problem
\begin{equation}
\setlength\arraycolsep{1.5pt}
\begin{array}{cc}
\mathrm{minimize} & f(x)\\
\mathrm{subject~to} & x \in X
\end{array} \label{cvx}
\end{equation}
where 
\begin{equation}
X =\{x\in \mathbb{R}^n|g_i(x)\leq 0,\forall^m_{i=1}\},\label{constr_set} 
\end{equation} is the domain of the problem \eqref{cvx}.
The functions $f:\mathbb{R}^n\rightarrow\mathbb{R}$, $g:\mathbb{R}^n \rightarrow \mathbb{R}^m$ are assumed to be continuously differentiable $(\mathcal{C}^2)$ with respect to $x$, with the following assumptions:
\begin{assumption}\label{ass1}
    $\nabla f: \mathbb{R}^n\rightarrow \mathbb{R}^n$ is strongly monotone on $X$, with $\mu>0$  such that the following holds:
    \begin{equation*}
    (x_1-x_2)^T(\nabla f(x_1)-\nabla f(x_2)) \geq \mu \norm{x_1-x_2}^2.
    \end{equation*}  
\end{assumption} 
As a consequence of Assumptions \ref{ass1}, it is derived that the objective function $f$ is strongly convex in $x$ with the modulus of convexity given by $\frac{\mu}{2}$.

\begin{assumption}\label{ass2}
    The constraint function $g(x)=Ax-b$ is linear in $x$.
\end{assumption} 
\begin{assumption}\label{ass3}
    There exists an $x \in \mathrm{relint}X$ such that $g_i(x)<0, \forall^{m}_{i=1}$.
\end{assumption}
\begin{assumption}\label{as4}
Let matrix $\frac{\partial g}{\partial x}$ have full row rank $m\leq n$ and $q_1I\leq \frac{\partial g}{\partial x}\frac{\partial g}{\partial x}^T \leq q_2I$, where $I$ is an identity matrix and $q_1,q_2$ are positive constants.
\end{assumption}
Assumptions \eqref{ass1}-\eqref{ass3} ensure that $x$ is strictly feasible and strong duality holds for the optimization problem \eqref{cvx}.

Let $L:\mathbb{R}^n \times \mathbb{R}^m \rightarrow \mathbb{R}$ define the \textit{Lagrangian} function of the optimization problem \eqref{cvx} as given below
\begin{equation}
L(x,\lambda)=f(x)+\lambda^Tg(x). \label{lg}
\end{equation}
Let $\lambda_i$ be the Lagrange multipliers associated with $g_i(x)$, then $\lambda\in \Lambda \subseteq \mathbb{R}^m_+=\{\lambda \in \mathbb{R}^m,\lambda_i \geq 0, \forall^m_{i = 1}\}$ defines the corresponding vectors of Lagrange multipliers.

The Lagrangian function $L$ defined in \eqref{lg} is $\mathbb{C}^2$-differentiable convex-concave in $x$ and $\lambda$ respectively, i.e., $L(.,\lambda)$ is convex for all $\lambda\in \Lambda$ and $L(x,.)$ is concave for all $x \in X$. 
We say that $(x^*,\lambda^*)$ is a saddle-point if the following holds:
\begin{align}
L(x^*,\lambda)\leq L(x^*,\lambda^*) \leq L(x,\lambda^*) \label{spp}
\end{align} for all $x\in X$ and $\lambda \in \Lambda$.

If $x^*$ is the unique minimizer of $L$, then it must satisfy the Karush-Kuhn-Tucker (KKT) conditions stated as follows.
\begin{align}
g_i(x^*)&\leq 0,~\forall^{m}_{i=1} \label{k1}\\
\lambda^*_i&\geq 0,~\forall^{m}_{i=1} \label{k2}\\
\lambda^*_i g_i(x^*)&=0,~\forall^{m}_{i=1} \label{k3}\\
\nabla f(x^*)+\lambda^{*T}\nabla g(x^*)&=0. \label{k5}
\end{align}

Let us define $z = (x,\lambda) \in \Omega = X \times \Lambda$, where $\Omega$ by definition is a nonempty and closed convex subset of $\mathbb{R}^n$ and $\mathbb{R}^m_{\geq 0}$. Then $z^*=(x^*,\lambda^*)$ is the saddle point solution of \eqref{lg}.

\begin{remark}
    Since strong duality holds, the KKT conditions \eqref{k1}-\eqref{k5} are necessary and sufficient to guarantee optimality of the problem \eqref{cvx}, with $x^*$ as the unique minimizer of \eqref{cvx} and $z^*$ as the unique saddle-point of \eqref{lg}.
\end{remark} 
Let $G:\mathbb{R}^n \times \mathbb{R}^m \rightarrow\mathbb{R}^{n+m}$ define the gradient map of \eqref{lg} as given below:
\begin{align}
G(z)=\nabla_z L=\begin{bmatrix}
-\nabla_x L(x,\lambda)\\
\nabla_\lambda L(x,\lambda)
\end{bmatrix} \label{gu}
\end{align}
Taking gradient descent and gradient ascent along the direction of $x$ and $\lambda$ variables, respectively, we propose an algorithm which in principal uses the framework of projected dynamical systems (for details, please refer to our online report \cite[Section 2]{bansode2019exponential} or Definition \ref{vi_Def}, Proposition \ref{fpppro}, and equation \eqref{frz} in the Appendix section). We designate the algorithm as projected primal-dual dynamics, it is as shown below:
\begin{equation}
\dot{z}=\beta\{P_{\Omega}[z-\alpha G(z)]-z\},\label{mpdd}
\end{equation}
where $\alpha,\beta>0$ are parameters adjusted to control stability and assure convergence, we set $\alpha=\beta=1$ to avoid confusion. $P_{\Omega}$ is a minimum norm projection operator of the form $P_{\Omega}=\arg \min_{v\in \Omega}\norm{z-v}$. Note that, there is no projection taken w.r.t. the primal variable $x$ as it belongs to $\mathbb{R}^n$ but only w.r.t. the dual variable $\lambda$ which is restricted to $\mathbb{R}^m_{\geq 0}$. The following property always holds for projection on $\Omega$,
\begin{align}
    [u-P_\Omega(u)]^T[P_\Omega(z)-z]\geq 0, \forall u\in\mathbb{R}^{n}\times \mathbb{R}^{m}_{\geq 0}, \forall z\in \Omega \label{proja}
\end{align} In what follows, we assess the stability of the proposed algorithm.

\subsection{Stability analysis}
Before proceeding to the stability analysis of the proposed dynamics \eqref{mpdd}, it is worth noting the Definition \ref{def1.1}-\ref{def1.3} and Proposition \ref{monoto} on the monotonicity property, stated in the Appendix of this paper.
\begin{lemma}\label{thm1}
    If Assumptions \eqref{ass1}-\eqref{ass3} hold, then $G(z)$ is monotone such that $[G(z_1)-G(z_2)]^T(z_1-z_2)\geq 0$
    for every pair of $z_1,z_2\in \Omega$.
\end{lemma}
\begin{proof}
    First we derive the Jacobian matrix of $G$ as
    \begin{equation}     
    \nabla G = \begin{bmatrix}
    \nabla^2 f(x) + \lambda^T\nabla^2 g(x) & \nabla g(x)^T\\
    -\nabla g(x)    & \mathbf{0}
    \end{bmatrix}
    \end{equation}
    It is know that $G$ is monotone if and only if the $\nabla G$ is positive semidefinite (see, \cite{karamardian1990seven}), which implies that $\nabla G$ must be positive semidefinite $\frac{1}{2}\nabla G+\frac{1}{2} \nabla G^T \geq 0,\forall z\in \Omega,\forall t$. We verify this property by evaluating the symmetric part of $\nabla G$:
    \begin{align}
    \frac{\nabla G + \nabla G^T}{2} &= \begin{bmatrix}
    \nabla^2 f(x) + \lambda^T\nabla^2 g(x) & \mathbf{0}\\
    \mathbf{0}    & \mathbf{0} 
    \end{bmatrix}\\
    &\geq {0},\forall z\in \Omega,\forall t\label{psd1}
    \end{align}
    This proves that $\nabla G(z)$ a is positive semi-definite matrix. Thus $G(z)$ is a monotone map.
\end{proof}
    
\begin{lemma}\label{thm3.5.1}
    Let ${G}(z)$ be continuously differentiable on an open convex subset of $\mathbb{R}^{n+m}$. If Assumptions \ref{ass1}-\ref{ass3} hold and $G$ is monotone for all $z\in \Omega$, then with $\alpha>0$, the projected PD dynamics \eqref{mpdd} is Lyapunov stable.
\end{lemma}
\begin{proof}
    For the Lagrangian function \eqref{lg}, the following inequalities always hold: $    {L}(x^*,\lambda^*)-{L}(x^*,\lambda)\geq 0$ and $    {L}(x,\lambda^*)-{L}(x^*,\lambda^*)\geq 0$.
    Let us define the Lyapunov function as follows:
    \begin{align}
    V(z) &= ({L}(x^*,\lambda^*)-{L}(x^*,\lambda))+({L}(x,\lambda^*)-{L}(x^*,\lambda^*))\nonumber\\
    &~~~+\frac{1}{2}\norm{z-z^*}^2.\label{Vz}
    \end{align}
    The last term in \eqref{Vz} ensures that $V(z) \geq \frac{1}{2}\norm{z-z^*}^2,\forall z \in \Omega$, thus also ensures the boundedness of the level sets of $V(z)$.
    
    Differentiating $V(z)$ along the trajectories of \eqref{mpdd} yields,
    \begin{align}
    \dot{V}(z)&=\nabla V(z)\dot{z}\nonumber\\
    &=-[(\nabla {L}(x,\lambda^*)-\nabla {L}(x^*,\lambda))+z-z^*]^T(z-\tilde{z})\nonumber\\
    &=-[G(z)+z-z^*]^T(z-\tilde{z}) \label{ff}
    \end{align}
    using \eqref{proja} with $u=z-\alpha G(z)$ and $z=z^*$, we get the following, \cite{gao2003exponential},
    \begin{equation}
    [z-z^*+\alpha G(z)]^T(z-\tilde{z})\geq \norm{z-\tilde{z}}^2+\alpha (z-z^*)^T{G}(z). \label{use1}
    \end{equation}
    Using \eqref{use1} in \eqref{ff} yields,
    \begin{align}
    \dot{V}(z)&\leq -(z-z^*)^T{G}(z)-\norm{z-\tilde{z}}^2\nonumber\\
    &\leq -(z-z^*)^T(G(z)-G(z^*))-\norm{z-\tilde{z}}^2 \label{ineq1}
    \end{align}
    Using Lemma \ref{thm1} in \eqref{ineq1}, we get $\dot{V}(z)\leq 0$. This proves that the projected dynamical system \eqref{mpdd} is stable in the sense of Lyapunov.
\end{proof}
We see that the results of Lemma \ref{thm1} and \ref{thm3.5.1} are not sufficient to prove exponential stability of the proposed dynamics \eqref{mpdd}. This however is overcome by adapting the framework of natural gradient \cite{amari1998natural}, which allows us to prove that the gradient $G(z)$ is strongly monotone. This property is later exploited to prove that the proposed dynamics is exponentially stable as discussed in the subsequent section.

\subsection{Geometry of the Primal-Dual Dynamics}
First we employ the framework of fiber bundles to understand the geometry of the gradient map \eqref{gu}. 
Assume that the proposed dynamics is embedded into a tuple $(X,\mathcal{M})$ with $X$ a manifold and $\mathcal{M}$ a fiber manifold above $X$, with projection $\Pi : \mathcal{M} \rightarrow X$. $X$ is considered as a state-space of primal variables while the fibers of $\mathcal{M}$ are along the space of the dual variables, $\Lambda$. If $x\in X \subset \mathbb{R}^n$, $\lambda\in \Lambda \subseteq \mathbb{R}^m_{\geq 0}$ then $\mathcal{M} = (X,\Lambda)$ with coordinates $(x,\lambda)$.
The tangent space of $\mathcal{M}$, denoted by $T_\mathcal{M}$ has coordinates $(\dot{x},\dot{\lambda})$.

Consider the function  
\begin{align}
    Q(x,\lambda) = \lambda -g(x).\label{nm5}
\end{align}
and the implicit surface $\mathcal{M}=\{x|\lambda - g(x) = 0\}$.
With \eqref{nm5}, if $x\in X$ the primal dynamics in \eqref{mpdd} on the manifold $\mathcal{M}$ reduces to:
\begin{align}
    \dot{x} = - \nabla f(x) - \Big( \frac{\partial g}{\partial x} \Big)^Tg(x). \label{onlyx}
\end{align}
It follows that \eqref{onlyx} is exponentially stable. 
\begin{lemma}
    The gradient dynamics \eqref{onlyx} is exponentially stable.
\end{lemma}
\begin{proof} Assume $g(x)= Ax-b$ and consider the following Krasovskii-type Lyapunov candidate function $V(x) = \frac{1}{2}\dot{x}^T\dot{x}$. Differentiating it along the trajectories of \eqref{onlyx}, we get
\begin{align}
    \dot{V}&=\dot{x}\Big(-\nabla^2 f(x)\dot{x}
    - \Big[\frac{\partial^2 g}{\partial^2 x}\Big]^Tg(x)\dot{x}
    - \Big[\frac{\partial g}{\partial x}\Big]^T\Big[\frac{\partial g}{\partial x}\Big]\dot{x}\Big)\\
    &\leq \dot{x}^T[-\nabla^2f(x)-k]\dot{x}\\
    &\leq -\gamma V(x)\label{vdotdotdot}
\end{align}
where $\gamma>0$ and $k$ is chosen as
\begin{align}
    k>\sqrt{q_2}.\label{kcond0}
\end{align}
\end{proof}
Global attractivity of the manifold $\mathcal{M}$ would ensure that the proposed dynamics in \eqref{mpdd} is globally exponentially stable but since it is known that this is not the case we study the geometry of the problem and identify conditions to improve the convergence rates to the fixed point of \eqref{mpdd}.

\subsubsection{Constructing a Riemannian metric}
Let $R$ define a semi-Riemannian metric on space $T_\mathcal{M}$ that endows a semi-Riemannian structure to the fiber-bundle $(x,\lambda)\in \mathbb{R}^{n \times m}$, as shown below:
\begin{align}
    R &= \nabla Q^T(x)\nabla Q(x)\\
    &=\begin{bmatrix}
    \big(\frac{\partial g}{\partial x}\big)^T\frac{\partial g}{\partial x} & \big(-\frac{\partial g}{\partial x}\big)^T\\
    -\frac{\partial g}{\partial x} & I
    \end{bmatrix}\label{psedoR}
\end{align} 
$R$ is a semi metric on the space $X \times \Lambda$ with connection $\big(\frac{\partial g}{\partial x}$.
The metric $R$ can be made symmetric positive definite by introducing in it the parameter $k$ such that $k$ scales the connection term and for $k=I$ the original connection as defined by the above semi-Riemann metric is obtained, i.e 
\begin{align}
    R =\begin{bmatrix}
    \big(\frac{\partial g}{\partial x}\big)^T\frac{\partial g}{\partial x} & \big(-\frac{\partial g}{\partial x}\big)^T\\
    -\frac{\partial g}{\partial x} & kI
    \end{bmatrix}.\label{rm1}
\end{align}

We are now in a position to develop an understanding of the geometry of the proposed dynamics on a Riemannian manifold $(\mathcal{M},R)$. The following definition will be useful in understanding the concept of a linear connection with respect to the manifold $(\mathcal{M},R)$,\cite{10.2307/j.ctt1bpm9t5}.
\begin{definition}
Let $\Pi: \mathcal{M} \rightarrow X$ be a smooth fiber bundle, a tangent vector $\nu \in T_{p}\mathcal{M},~p \in \mathcal{M}$, is said to be vertical if $\Pi_{x_p}(\nu) = 0$. $V(p)$ denotes the set of all vertical tangent vectors in $P$. A distribution $H$ on $\mathcal{M}$ is said to be horizontal if $T_p\mathcal{M} = V_p \oplus H(p)$ for all $p \in \mathcal{M}$. 
\end{definition}
\begin{remark}
If $H$ is horizontal, it implies that for all $p\in \mathcal{M}$, $H(p)$ is a linear subspace of $T_p\mathcal{M}$ with the following properties:
\begin{align}
    \dim H(p) &= \dim X\\
    H(p) \cap \mathcal{M}(p) &= \Phi
\end{align}
$\Pi_{X_p}$ maps $H(p)$ isometrically onto $T_{\Pi(p)}X$.
\end{remark}
A connection in the bundle $\mathcal{M}$ is due to a unique splitting scheme of the tangent space $T_p\mathcal{M}$ into a horizontal and vertical space as shown in Remark \ref{remark33} in the Appendix section. The preferred direction of the vertical vector is along the fibers of $\mathcal{M}$. 

\subsubsection{Strongly monotone gradient of the Lagrangian}\label{3.a}
The natural gradient of $L$ at $z\in \mathcal{M}$ is a unique tangent vector $\mathrm{grad} L$ given as
\begin{equation}
\langle \mathrm{grad} L,v\rangle_r=D_zL(v),\forall v\in T_z\mathcal{M}. \label{lgre}
\end{equation}
In the matrix notation, \eqref{lgre} implies the following
\begin{equation}
\mathrm{grad}_r L = R^{-1}\nabla L^T \label{gradL}
\end{equation}
where $\nabla L = G(z)$ is the gradient vector of $L$ on Euclidean space $\mathbb{R}^{m+n}$. 

Denote $G_r(z)=\mathrm{grad}_rL$, the linear map $\mathbb{H}_{G_r(z)}:T_z\mathcal{M}\rightarrow\mathcal{M}$ assigned to each point $z\in \mathcal{M}$ is defined by the Hessian of $L$, denoted by $\mathbb{H}_{G_r}(z)v=\nabla_v G_r(z)=R^{-1}\nabla G(z),\forall v\in T_z\mathcal{M}$.

The projection operator $P^r_\mathcal{M}:\mathbb{R}^{n+m}\rightarrow \mathcal{M}$ defined as
\begin{align*}
P^r_\mathcal{M}(z)=\arg\min_{v\in T_z\mathcal{M}}\norm{z-v}^2_r.
\end{align*}
Correspondingly, the projected PD dynamics on $\mathcal{M}$ is defined as follows:
\begin{equation}
\dot{z}=\beta\{P^r_\mathcal{M}[z-\alpha G_r(z)]-z\}.\label{mpddr}
\end{equation}

Replacing $g(x)$ by $Ax-b$ as defined in Assumption \ref{ass2}, we define the gradient vector $G_r(z) \in T_z\mathcal{M}$ as follows:
\begin{align}
G_r(z) &= R^{-1}G(z),\nonumber\\
&=\begin{bmatrix}
k\nabla f(x)-A^TAx+kA^T\lambda+A^Tb\\
A\nabla f(x)-kAx+AA^T\lambda+kb
\end{bmatrix}\label{newgrad}
\end{align} 

In the following section, it is proved that the gradient map \eqref{newgrad} is strongly monotone.
\begin{proposition}\label{thm3.4}
    Consider the optimization problem \eqref{cvx} and let $(\mathcal{M},r)$ be a $n+m$-dimensional smooth manifold. If Assumption \ref{ass1} and \ref{as4} hold for the problem \eqref{cvx}, then with the linear map $R^{-1}:T_z\mathcal{M}\rightarrow T_z\mathcal{M}$, the gradient vector $G_r(z)$ is strongly monotone.
\end{proposition}
\begin{proof}
    For $G_r(z)$ to be strongly monotone, $\nabla_z {G}_r$ must be  positive definite\cite{karamardian1990seven}, i.e., for the symmetric part of $\nabla_z {G}_r$, i.e. $\frac{1}{2}\nabla {G}_r+\frac{1}{2}\nabla {G}^T_r$, the following must hold:
    \begin{align} 
    \nabla {G}_r+\nabla {G}^T_r&=R^{-1}\nabla G+\nabla G^TR^{-1},\nonumber\\
    &\geq \nu \mathrm{I},\forall z \in \mathcal{M},\forall t\label{gradg}
    \end{align} where $\nu>0$ is a constant, $\mathrm{I}$ is an identity matrix of appropriate dimensions.
    
    The Jacobian of $G_r(z)$, denoted by $\nabla G_r(z)$ is given below:
    \begin{align*}     
    \nabla G_r(z) = \begin{bmatrix}
    k\nabla^2 f(x) - A^TA & kA^T\\
    A\nabla^2f(x)-kA    & AA^T
    \end{bmatrix}
    \end{align*}
    The symmetric part of $\nabla G_r(z)$ is obtained as:
    \begin{align}     
    &\frac{\nabla {G}_r(z)+\nabla {G}^T_r(z)}{2}\nonumber \\= &\begin{bmatrix}
    k\nabla^2 f(x) - A^TA & \frac{1}{2}(A\nabla^2f(x))^T\\
    \frac{1}{2}A\nabla^2f(x)    & AA^T
    \end{bmatrix}\label{nablaGr}
    \end{align}
    Let $\mathrm{M}=\nabla {G}_r(z)+\nabla {G}^T_r(z)-q_1\mathrm{I}>0$. Then
    \begin{align}
    \mathrm{M} &= \begin{bmatrix}
    2k\nabla^2 f(x) - 2A^TA -q_1\mathrm{I} & (A\nabla^2f(x))^T\\
    A\nabla^2f(x)    & 2AA^T-q_1\mathrm{I}
    \end{bmatrix}\nonumber\\
    &\geq \begin{bmatrix}
    2k\nabla^2 f(x) - 2A^TA -q_1\mathrm{I} & (A\nabla^2f(x))^T\\
    A\nabla^2f(x)    & AA^T
    \end{bmatrix}\label{Ma}.
    \end{align} 
    Further let $\mathrm{S}=AA^T$, then the Schur compliment of the block $\mathrm{S}$ of the matrix $\mathrm{M}$, denoted by $\mathrm{S}_{\mathrm{Schur}}$ is derived as
    \begin{align} 
    \mathrm{S}_{\mathrm{Schur}} &= 2k\nabla^2 f(x) - 2A^TA-q_1I\nonumber\\
    &-(A\nabla^2f(x))^T(AA^T)^{-1}A\nabla^2f(x).\label{m/s}
    \end{align}
    Let $\mathrm{H}=\nabla^2f(x)$ for the notational simplicity. Note that in \eqref{pqrs}, $2k\mathrm{H}>0,\forall k>0$, $2A^TA \geq 0$, $q_1I>0$, and $0\leq \mathrm{H}A^T(AA^T)^{-1}A\mathrm{H} \leq \mathrm{H}^2$. The last terms is a consequence of $A^T(AA^T)^{-1}A\leq I$. Rearranging \eqref{m/s} as given below 
    \begin{align}
    2k\mathrm{H} &> 2A^TA+q_1I+\mathrm{H}A^T(AA^T)^{-1}A\mathrm{H}\nonumber\\
    2k\mathrm{H} &> 2A^TA+q_1I+\mathrm{H}^2\label{pqrs}
    \end{align}
    allows to choose $k$ such that $\mathrm{S}_{\mathrm{Schur}}>0$.
    Post multiplying \eqref{pqrs} by $(2H)^{-1}$ yields the following:
    \begin{align}
    2kI &>  2A^TA(2H)^{-1}+q_1(2H)^{-1}+\mathrm{H}^2(2H)^{-1}\nonumber\\
    kI &> A^TAH^{-1}+\frac{1}{2}q_1H^{-1}+\frac{1}{2}\mathrm{H}\label{kcond}.\end{align}
    
    Applying Courant-Fischer theorem \cite{horn1990matrix} to \eqref{kcond} yields the following:
    \begin{equation}
    \lambda_{max}(kI)>\lambda_{max}(A^TAH^{-1}+\frac{1}{2}q_1H^{-1}+\frac{1}{2}\mathrm{H}).\label{mkcond}
    \end{equation}
    Since $\lambda_{max}(kI) = k$, \eqref{mkcond} has the following form:
    \begin{equation}
    k>\lambda_{max}(A^TAH^{-1}+\frac{1}{2}q_1H^{-1}+\frac{1}{2}\mathrm{H}).\label{mkcond1}
    \end{equation}
    By choosing $k$ as given in \eqref{mkcond1} ensures that $\mathrm{S}_{\mathrm{Schur}}>0$. 
    But $k$ must also satisfy \eqref{kcond0}, thus $k$ must be chosen such that the following holds:
    \begin{equation}
    k > \max\{\sqrt{q_2},\lambda_{max}(A^TAH^{-1}+\frac{1}{2}q_1H^{-1}+\frac{1}{2}\mathrm{H})\} \label{choosek}
    \end{equation}
    ensures that both \eqref{kcond0} and \eqref{mkcond} are met. If $k$ is chosen according to \eqref{choosek}, then $\mathrm{S}_{\mathrm{Schur}}>0$ holds such that there exists a $\nu \geq \frac{q_1}{2}$ which implies that 
    \begin{equation}
    \langle \mathbb{H}_{G_r(z)}v,v\rangle_r \geq \nu\norm{v}^2_r,\forall v \in T_z\mathcal{M}.
    \end{equation}        
    Hence it follows that
    \begin{equation}
    \langle {G}_r(z_1)-{G}_r(z_2),z_1-z_2\rangle_r\geq \nu \norm{z_1-z_2}^2_r. \label{ffinal}
    \end{equation}
    Hence it is proved that $G_r(z)$ is strongly monotone.
\end{proof}

\subsubsection{Exponential stability}\label{3.b}
Without loss of generality, let us define $G_r(z)$ similar to \eqref{gu} as follows:
\begin{equation}
G_r(z)=\begin{bmatrix}
\nabla^r_x {L}(x,\lambda)\\
-\nabla^r_\lambda {L}(x,\lambda)
\end{bmatrix}, \label{gut}
\end{equation} where ${L}(x,\lambda)$ would represent the modified Lagrangian function whose gradient vector field is given by $G_r(z)$. Since, $G_r(z)$ is strongly monotone on $\mathcal{M}$, \eqref{mpddr} will converge to a unique saddle-point solution $z^*$.
\begin{theorem}\label{thm3.5}
    Let $G_r(z)$ be Lipschitz continuous on an open set including $\mathcal{M}$, then inequality \eqref{ffinal} and $\alpha>0$, imply that the system \eqref{mpddr} with $z(0) \in \mathcal{M}$ is globally exponentially stable at the unique solution $z^*$ of \eqref{spp}.
\end{theorem}
\begin{proof}
    For each $z(0) \in \mathcal{M}$, there exists a unique solution $z(t)$ of \eqref{mpdd}, that started from $z(0)$. If $[0,t_f)$ is the maximal interval of $z(t)$, then from Lemma \ref{lemma1.3}, $z(t) \in \mathcal{M}$ for all $t \in [0,t_f)$. Since $G_r(z)$ is strongly monotone, the following holds: $
    L(x^*,\lambda^*)-L(x^*,\lambda)> 0,
    L(x,\lambda^*)-L(x^*,\lambda^*)> 0$
    Let us define the Lyapunov function for the dynamics \eqref{mpddr} as follows:
    \begin{align}
    V_1(z) &= (L(x^*,\lambda^*)-L(x^*,\lambda))+(L(x,\lambda^*)-L(x^*,\lambda^*))\nonumber\\
    &~~~+\frac{1}{2}\norm{z-z^*}^2_r.
    \end{align}
    It is to be noted that $V_1(z)$ possesses a similar structure as that of $V(z)$ defined in \eqref{Vz}, it is also differentiable convex on $\mathcal{M}$, with $V_1(z) \geq \frac{1}{2}\norm{z-z^*}^2_r,\forall z \in \mathcal{M}$, thus bounding all level sets of $V_1(z)$.
    
    Differentiating $V_1(z)$ along the trajectories of \eqref{mpddr} yields:
    \begin{align}
    \dot{V}_1(z)&=\nabla V_1(z)\dot{z}\nonumber\\
    &=-\langle\nabla L(x,\lambda^*)-\nabla L(x^*,\lambda)+z-z^*,z-\tilde{z}\rangle_r\nonumber\\
    &=-\langle G_r(z)+z-z^*,z-\tilde{z}\rangle_r \label{eff}
    \end{align}
    Substituting $u = z-\alpha G_r(z)$ and $z=z^*$ in \eqref{proja}, yields
    \begin{equation}
    \langle z-z^*+\alpha G_r(z),z-\tilde{z}\rangle_r\geq \norm{z-\tilde{z}}^2_r+\langle\alpha (z-z^*),G_r(z)\rangle_r. \label{euse1}
    \end{equation}
    Using \eqref{euse1} in \eqref{eff} yields,
    \begin{equation}
    \dot{V}_1(z)\leq -\langle\alpha (z-z^*),G_r(z)\rangle_r.\label{V1zdot}
    \end{equation}
    If $k$ is chosen such that the condition \eqref{mkcond} is satisfied then $G_r(z)$ is strongly monotone. The strong monotonicity of $G_r(z)$ leads to the following property of the Lagrangian function $L(z),\forall z\in \mathcal{M}$,
    \begin{align}
    &\langle z-z^*,G_r(z)\rangle_r\nonumber\\ &\geq L(x,\lambda^*)-L(x^*,\lambda)+\frac{\nu}{2}\norm{z-z^*}^2_r, z \in \mathcal{M}.\label{strong conv}
    \end{align}
    
    Using \eqref{strong conv}, \eqref{V1zdot} modifies to the following
    \begin{align}
    \dot{V}_1(z)&\leq -\langle\alpha (z-z^*),G_r(z)\rangle_r\nonumber,\\
    &\leq-\alpha \beta[L(x,\lambda^*)-L(x^*,\lambda)+\frac{\nu}{2}\norm{z-z^*}^2_r],\nonumber\\
    &\leq-\alpha \beta[(L(x^*,\lambda^*)-L(x^*,\lambda))\nonumber\\
    &~~~+(L(x,\lambda^*)-L(x^*,\lambda^*))+\frac{\nu}{2}\norm{z-z^*}^2_r].
    \end{align}
    With $\alpha,\beta>0$, it can be shown that,
    \begin{equation*}
    \dot{V}_1(z)\leq -\beta\min\{1,\alpha\nu\}V(z).
    \end{equation*}
    Thus, it is proved that the system \eqref{mpdd} is exponentially stable at the unique solution $z^*$ of \eqref{spp}.
    Therefor,
    \begin{equation*}
    \norm{z-z^*}_r\leq ce^{-\beta\frac{\min\{1,\alpha\nu\}}{2}t}
    \end{equation*}
    where $c=\sqrt{2V_1(z(0))}$.
    
    Further, if $G_r(z)$ is Lipschitz continuous on $\mathcal{M}$, i.e., $\norm{{G}_r(z_1)-{G}_r(z_2)}_r\leq \mathbf{L}\norm{z_1-z_2}_r,\forall z_1,z_2\in \mathcal{M}$, where $\mathbf{L}$ is a Lipschitz constant then by using \cite[Theorem 4]{gao2003exponential} the global exponential stability of the projected PD dynamics can be established as follows:
    \begin{equation}
    \norm{z(t)-z^*}_r\leq \norm{z(0)-z^*}_re^{\frac{-\alpha \beta(4\nu-\alpha \ell^2)}{8}t},\forall t \geq 0. \label{conva}
    \end{equation}
    If $\alpha<\frac{4\nu}{\mathbf{L}^2}$, it follows that the projected PD dynamics \eqref{mpddr} is globally exponentially stable.
\end{proof}
\section{Simulation Results}\label{sim}
This section presents simulation studies of the projected PD dynamics \eqref{mpddr}. It is known that the Euler discretization of the exponentially stable dynamical system owns geometric rate of convergence \cite{stuart1994numerical} for sufficiently small step-sizes. The projected PD dynamics \eqref{mpddr} is Euler discretized with a step size $s>0$ and the following discrete-time projected PD dynamics\cite{nagurney2012projected} is obtained.
\begin{equation}
z(\tau+1)=\beta P^r_\mathcal{M}\{z(\tau)-\alpha G_r(\tau)\}.
\end{equation}

First example (Example 1) considers an optimization problem of the form \eqref{cvx} with $m=5$ and $n=10$. The Hessian matrix is assumed to be $H = 20{I}$ with $A$ and $b$ taken as Gaussian random matrix and vector respectively. The distance to the primal optimizer $x^*$ for different values of parameter $k$ is shown in Fig. \ref{eps1}, where $\varrho=\max\{\sqrt{q_2},\lambda_{max}(A^TAH^{-1}+\frac{1}{2}q_1H^{-1}+\frac{1}{2}\mathrm{H})\}$. It can be seen from the plot that the rate of convergence to the equilibrium point accelerates as the value of $k$ is increased. It implies that increasing the value of $k$ allows increasing the value of $\nu$, which further increases the coefficient of the negative exponential term in \eqref{conva}. 

In the second example, an $L_2$ regularized least squares problem is considered with $m=30$ and $n=50$. The objective function is $f(x) = \norm{Cx-d}^2_2+\frac{\theta}{2}\norm{x}^2_2$ with $\theta>0$, constrained to $Ax\leq b$. Matrices $(C,A) \in \mathbb{R}^{m \times n}$, and vectors $(d,b)\in \mathbb{R}^{m \times 1}$ are Gaussian random matrices and vectors, respectively. Parameters $\alpha,\beta$ are chosen as unity and the proposed dynamics \eqref{mpddr} is simulated for $k = 1000\max(\varrho)$. A sketch of the error norm as a function of time is shown in Fig. \ref{eps3}. It can be seen that the error norm $\norm{x_i-x^*_i}^2$ has geometric rate of convergence.
\begin{figure}[t]
    \centering
    \includegraphics[width=2.5in]{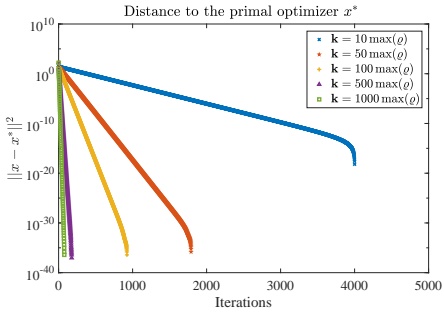}
    \caption{Distance to the primal optimizer $x^*$ for different values of $k$ (Example 1).}
    \label{eps1}
\end{figure}
\begin{figure}[t]
    \centering
    \includegraphics[width=2.5in]{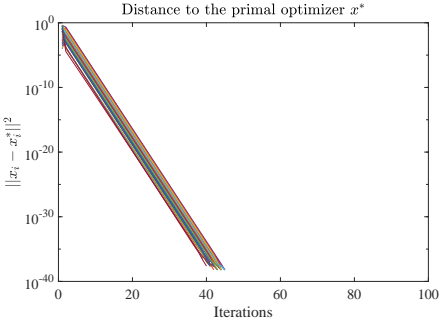}
    \caption{Distance to the primal optimizer $x^*$ ($L_2$ regularized least squares problem).}
    \label{eps3}
\end{figure}
\section{Conclusions and discussion}\label{concl} In this paper we proposed a Riemannian geometric framework with natural gradient adaptation to achieve exponentially convergent projected primal-dual dynamics when applied to linear inequality constrained optimization problems. We began by framing the proposed dynamics in a fiber-bundle setting endowed with a Riemannian metric $R$ that captures the geometry of the gradient vector. The metric $R$ induced a unique decomposition of the target space into a horizontal and vertical distributions. The natural gradient proved to be strongly monotone on $\mathcal{M}$ leading to an exponentially stable saddle-point solution. We further showed that the increasing values of $k$ result in much steeper gradient that leads to an accelerated convergence to the saddle-point solution. 


\bibliographystyle{unsrt}        
\bibliography{references}  


\section*{Appendix} 
\begin{definition}(The Variational Inequality Problem, \cite{nagurney2012projected})\label{vi_Def}\\
    For a closed convex set $X\in \mathbb{R}^n$ and vector function $F:X\rightarrow \mathbb{R}^n$, the finite dimensional variational inequality problem, $\mathrm{VI(F,X)}$, is to determine a vector $x^*\in X$ such that 
    \begin{equation}
    (x-x^*)^TF(x^*)\geq 0,~\forall x\in X. \label{vip}
    \end{equation}
\end{definition}
A variational inequality problem \eqref{vip} is equivalent to a fixed point problem given below:
\begin{proposition}(A Fixed Point Problem,\cite{nagurney2012projected})\label{fpppro}\\
    $x^*$ is a solution to $\mathrm{VI(F,X)}$ if and only if for any $\alpha >0$, $x^*$ is a fixed point of the projection map:
    \begin{equation}
    x^*=P_X(x^*-\alpha F(x^*)) \label{fppprob}
    \end{equation}
    where \begin{equation}
    P_{X}=\arg \min_{v\in X}\norm{x-v}.\label{projection}
    \end{equation}
\end{proposition}

\begin{theorem}(Uniqueness of the Solution to Variational Inequality, \cite{nagurney2012projected})\label{uniqueo}\\
    Suppose that $F(x)$ is strongly monotone on $X$. Then there exists precisely one solution $x^*$ to $\mathrm{VI(F,X)}$.
\end{theorem}

Consider the following globally projected dynamical system proposed in \cite{friesz1994day}:
\begin{equation}
\dot{x}=\beta\{P_{X}[x-\alpha F(x)]-x\}\label{frz}
\end{equation}
where $\beta,\alpha$ are positive constants and $P_{X}:\mathbb{R}^{n}\rightarrow X$ is a projection operator as defined in \eqref{projection}.

\begin{remark}(\cite{gao2003exponential})\label{remark303}\\
    $x^*$ is an equilibrium point of \eqref{frz} if and only if $x^*$ is a solution of the variational inequality problem \eqref{vip}.
\end{remark}
From Remark \ref{remark303}, 
\begin{equation*}
\dot{x}=0\implies x^* = P_X(x^*-\alpha F(x^*)).
\end{equation*}
\begin{definition}
{(Monotone Map,\cite{karamardian1990seven})}\label{def1.1}\\
    A mapping $F$ is monotone on $X \subseteq \mathbb{R}^n$, if for every pair of distinct points $x,y\in X$, we have 
    \begin{equation*}
    (y-x)^T(F(y)-F(x)) \geq 0.
    \end{equation*}  
\end{definition}

\begin{definition}{(Strongly Monotone Map,\cite{karamardian1990seven})}\label{def1.3}\\
    A mapping $F$ is strongly monotone on $X \subseteq \mathbb{R}^n$, if there exists $\mu>0$ such that, for every pair of distinct points $x,y\in X$, we have 
    \begin{equation*}
    (y-x)^T(F(y)-F(x)) \geq \mu \norm{x-y}^2.
    \end{equation*}  
\end{definition}

The relation between monotonicity of $F$ and positive definiteness of its Jacobian matrix
\begin{align*}
\nabla F(x)=\Bigg (\frac{\partial F_i(x)}{\partial x_j}\Bigg)_{i,j=1,2,\ldots,n},
\end{align*}
as given below.

\begin{proposition}((Strongly) Positive Definite Jacobian of $F(x)$ implies (Strongly) Monotone $F(x)$,\cite{nagurney2012projected})\label{monoto}\\
    Suppose that $F$ is continuously differentiable on $X$.
    \begin{enumerate}
        \item If the Jacobian matrix $\nabla F(x)$ is positive semidefinite, i.e., 
        \begin{align*}
        y^T \nabla F(x) y \geq 0, \forall y \in \mathbb{R}^n, x \in X,
        \end{align*}
        then $F$ is monotone on $X$.
        \item If the Jacobian matrix $\nabla F(x)$ is positive definite, i.e., 
        \begin{align*}
        y^T \nabla F(x) y > 0, \forall y \in \mathbb{R}^n, x \in X,
        \end{align*}
        then $F$ is strictly monotone on $X$.
        \item If $\nabla F(x)$ is strongly positive definite, i.e., 
        \begin{align*}
        y^T\nabla F(x) y \geq \mu \norm{z}^2,\forall y \in \mathbb{R}^n, x \in X,
        \end{align*}
        then $F(x)$ is strongly monotone on $X$.
    \end{enumerate}    
\end{proposition}

\subsubsection{Splitting of the tangent vector}
Consider the tangent vector $(\dot{x},\dot{\lambda})\in \mathbb{R}^{n+m}$, with the Euclidean metric $\mathbb{I}$, then 

$v_H=(\dot{x},0)$ and $v_V=(0, \dot{\lambda})$ such that
\begin{align}
    \langle v_H,v_V \rangle_I=\begin{bmatrix}
    \dot{x} & \mathrm{0}
    \end{bmatrix}\begin{bmatrix}I & \mathrm{0}\\ \mathrm{0} & I\end{bmatrix}\begin{bmatrix}
    \mathrm{0}\\\dot{\lambda}
    \end{bmatrix}
\end{align}
i.e., $v_H\perp v_V$ and $(\dot{x},\dot{\lambda}) = v_H \oplus v_V=0$.

\begin{remark}\label{remark33} Given a tangent vector $v = (\dot{x},\dot{\lambda}) \in T_\mathcal{M}$, orthogonality of $v_V$ and $v_H$ is preserved under the new metric $R$.
\end{remark}
\begin{proof}
From \eqref{psedoR}, we observe that the splitting of the vectors is only dependent on $m_{12}$ or $m_{21}$, and $m_{22}$ in any case.
Let us rewrite $R$ as 
    $R = \begin{bmatrix}
    m_{11} & m_{12} \\ m_{21} & m_{22}
    \end{bmatrix}$.
Now, the relevant splitting of the tangent vectors $(\dot{x}, \dot{\lambda})$ confirming the structure of $R$ defined in \eqref{rm1} is as follows:
\begin{align}
    (\dot{x},\dot{\lambda}) &= v_H\oplus v_V\\
    &=\Big( \dot{x}, -{m^{-1}_{22}}{m_{21}}\dot{x}\Big)\oplus\Big( 0, \dot{\lambda}+{m^{-1}_{22}}{m_{21}}\dot{x}\Big)\\
    &=\begin{bmatrix}\dot{x} & -{m^{-1}_{22}}{m_{21}}\dot{x}\end{bmatrix}\begin{bmatrix}m_{21}\Big(\dot{\lambda}+{m^{-1}_{22}}{m_{21}}\dot{x}\Big)\\
    m_{22}\Big(\dot{\lambda}+{m^{-1}_{22}}{m_{21}}\dot{x}\Big)
    \end{bmatrix}\\
    &=m_{21}\dot{x}\Big(\dot{\lambda} + {m^{-1}_{22}}{m_{21}}\dot{x}\Big) - m_{21}\dot{x}\Big(\dot{\lambda}+{m^{-1}_{22}}{m_{21}}\dot{x}\Big) \label{the64}\\
    &= 0 
\end{align}
\end{proof}
\begin{remark}
As the value of parameter $k$ increases the vertical component in $v_H$, i.e., $-{m^{-1}_{22}}{m_{21}}\dot{x}$ decreases in value, which causes the trajectories off-manifold $\mathcal{M}$ to approach $\mathcal{M}$ faster.
\end{remark}
\begin{lemma}\label{lemma1.3}{\cite{xia2000stability}}
    Assume that $F$ is locally Lipschitz continuous in a domain $D$ that contains $X$. Then the solution $x(t)$ of \eqref{frz} will approach exponentially the feasible set $X$ when the initial point $x^0\notin X$. Moreover, if $x^0\in X$, then $x(t) \in X$.
\end{lemma}

\end{document}